\documentclass{amsart}
\usepackage{amsfonts}

\newtheorem{theorem}{Theorem}
\theoremstyle{plain}

\newtheorem{corollary}{Corollary}

\newtheorem{definition}{Definition}
\newtheorem{example}{Example}

\newtheorem{lemma}{Lemma}

\numberwithin{equation}{section}
\input{tcilatex}

\begin{document}
\title{Characterizations of intra-regular $\Gamma $-AG$^{\ast \ast }$%
-groupoids by the properties of their $\Gamma $-ideals}
\author{$^{1}$\textbf{Madad Khan, }$^{2}$\textbf{Venus Amjid} \textbf{and }$%
^{3}$\textbf{Faisal }}
\maketitle

\begin{center}
\textbf{Department of Mathematics}

\textbf{COMSATS Institute of Information Technology}

\textbf{Abbottabad, Pakistan.}\bigskip

$^{1}$\textbf{E-mail: madadmath@yahoo.com}

$^{2}$\textbf{E-mail: venusmath@yahoo.com}

$^{3}$\textbf{E-mail: yousafzaimath@yahoo.com}

\bigskip
\end{center}

\textbf{Abstract.} We have characterized an intra-regular $\Gamma $-AG$%
^{\ast \ast }$-groupoids by using the properties of $\Gamma $-ideals (left,
right, two-sided ), $\Gamma $-interior, $\Gamma $-quasi, $\Gamma $-bi and $%
\Gamma $-generalized bi and $\Gamma $-$(1,2))$. We have prove that all the $%
\Gamma $-ideals coincides in an intra-regular $\Gamma $-AG$^{\ast \ast }$%
-groupoids. It has been examined that all the $\Gamma $-ideals of an
intra-regular $\Gamma $-AG$^{\ast \ast }$-groupoids are $\Gamma $%
-idempotent. In this paper we define all $\Gamma $-ideals in $\Gamma $-AG$%
^{\ast \ast }$-groupoids and we generalize some results.

\textbf{Keywords}. $\Gamma $-AG-groupoid, intra-regular $\Gamma $-AG$^{\ast
\ast }$-groupoid and $\Gamma $-$(1,2)$-ideals.

\begin{center}
\bigskip

{\LARGE Introduction}
\end{center}

The idea of generalization of commutative semigroup was introduced in $1972$%
, they named it as left almost semigroup (LA-semigroup in short)$($see \cite%
{kaz}$)$. It is also called an Abel-Grassmann's groupoid (AG-groupoid in
short) \cite{ref10}. In this paper we will call it an AG-groupoid.

This structure is closely related with a commutative semigroup because if an
AG-groupoid contains a right identity, then it becomes a commutative monoid 
\cite{Mus3}. A left identity in an AG-groupoid is unique \cite{Mus3}. It is
a mid structure between a groupoid and a commutative semigroup with wide
range of applications in theory of flocks \cite{Naseeruddin}. Ideals in
AG-groupoids have been discussed in \cite{Mus3}, \cite{O.Steinfeld}, \cite%
{mkhan} and \cite{myousaf}. In 1981. the notion of $\Gamma $-semigroups was
introduced by M. K. Sen \cite{gemma1} and \cite{gamma2}

In this paper, we have introduced the notion of $\Gamma $-AG$^{\ast \ast }$%
-groupoids. $\Gamma $-AG-groupoids is the generalization of $\Gamma $%
-AG-groupoids. Here, we explore all basic $\Gamma $-ideals, which includes $%
\Gamma $-ideals (left, right,two-sided ), $\Gamma $-interior, $\Gamma $%
-quasi, $\Gamma $-bi, $\Gamma $-generalized bi and $\Gamma $-$(1,2)).$

\begin{definition}
Let $S$ and $\Gamma $ be two non-empty sets, then $S$ is said to be a $%
\Gamma $-AG-groupoid if there exist a mapping $S\times \Gamma \times
S\rightarrow S$, written $\left( x\text{, }\gamma \text{, }y\right) $ as $%
x\gamma y$, such that $S$ satisfies the left invertive law, that is%
\begin{equation}
\left( x\gamma y\right) \delta z=\left( z\gamma y\right) \delta x,\text{ for
all }x,y,z\in S\text{ and }\gamma ,\delta \in \Gamma .  \tag{1}
\end{equation}
\end{definition}

\begin{definition}
\bigskip \ A $\Gamma $-AG-groupoid $S$ is called a $\Gamma $-medial if it
satisfies the medial law, that is 
\begin{equation}
\left( x\alpha y\right) \beta \left( l\gamma m\right) =\left( x\alpha
l\right) \beta \left( y\gamma m\right) ,\text{ for all }x,y,l,m\in S\text{
and }\alpha ,\beta ,\gamma \in \Gamma  \tag{2}
\end{equation}
\end{definition}

\begin{definition}
A $\Gamma $-AG-groupoid $S$ is called a $\Gamma $-AG$^{\ast \ast }$-groupoid
if it satisfy the following law%
\begin{equation}
a\alpha (b\beta c)=b\alpha (a\beta c),\text{ for all }a,b,c\in S\text{ and }%
\alpha ,\beta \in \Gamma .  \tag{3}
\end{equation}
\end{definition}

\begin{definition}
A $\Gamma $-AG-groupoid$^{\ast \ast }$ $S$ is called a $\Gamma $-paramedial
if it satisfies the paramedial law, that is 
\begin{equation}
\left( x\alpha y\right) \beta \left( l\gamma m\right) =\left( m\alpha
l\right) \beta \left( y\gamma x\right) ,\text{ for all }x,y,l,m\in S\text{
and }\alpha ,\beta ,\gamma \in \Gamma .  \tag{4}
\end{equation}
\end{definition}

\begin{definition}
\label{1 copy(1)}Let $S$ be a $\Gamma $-AG-groupoid, a non-empty subset $A$
of $S$ is called $\Gamma $-AG-subgroupoid if $a\gamma b\in A$ for all $a$, $%
b\in A$ and $\gamma \in \Gamma $ or $A$ is called $\Gamma $-AG-subgroupoid if%
$\ A\Gamma A\subseteq A.$
\end{definition}

\begin{definition}
A subset $A$ of a $\Gamma $-AG-groupoid $S$ is called left(right) $\Gamma $%
-ideal of $S$ if $S\Gamma A\subseteq A\left( A\Gamma S\subseteq A\right) $
and $A$ is called $\Gamma $-ideal of $S$ if it is both left and right $%
\Gamma $-ideal.
\end{definition}

\begin{definition}
A $\Gamma $-AG-subgroupoid $A$ of a $\Gamma $-AG-groupoid $S$ is called a $%
\Gamma $-bi-ideal of $S$ if $\left( A\Gamma S\right) \Gamma A\subseteq A$.
\end{definition}

\begin{definition}
A $\Gamma $-AG-subgroupoid $A$ of a $\Gamma $-AG-groupoid $S$ is called a $%
\Gamma $-interior ideal of $S$ if $\left( S\Gamma A\right) \Gamma S\subseteq
A.$
\end{definition}

\begin{definition}
A $\Gamma $-AG-groupoid $A$ of a $\Gamma $-AG-groupoid $S$ is called a $%
\Gamma $-quasi-ideal of $S$ if $S\Gamma A\cap A\Gamma S\subseteq A.$
\end{definition}

\begin{definition}
A $\Gamma $-AG-subgroupoid $A$ of a $\Gamma $-AG-groupoid $S$ is called a $%
\Gamma $-$(1,2)$-ideal of $S$ if $\left( A\Gamma S\right) \Gamma
A^{2}\subseteq A$.
\end{definition}

\begin{definition}
A $\Gamma $-ideal $P$ of a $\Gamma $-AG-groupoid $S$ is called $\Gamma $%
-prime$\left( \Gamma \text{-semiprime}\right) $ if for any $\Gamma $-ideals $%
A$ and $B$ of $S$, $A\Gamma B\subseteq P\left( A\Gamma A\subseteq P\right) $
implies either $A\subseteq P$ or $B\subseteq P\left( A\subseteq P\right) .$
\end{definition}

\begin{definition}
An element $a$ of an $\Gamma $-AG-groupoid $S$ is called an intra-regular if
there exists $x,y\in S$ and $\beta ,\gamma ,\delta \in \Gamma $ such that $%
a=(x\beta (a\delta a))\gamma y$ and $S$ is called an intra-regular $\Gamma $%
-AG-groupoid $S$, if every element of $S$ is an intra-regular.
\end{definition}

\begin{example}
Let $S$ and $\Gamma $ be two non-empty sets, then $S$ is said to be a $%
\Gamma $-AG-groupoid if there exist a mapping $S\times \Gamma \times
S\rightarrow S$, written $\left( x\text{, }\gamma \text{, }y\right) $ as $%
x\gamma y$, such that $S=S=\{a,b,c,d,e\}$
\end{example}

\begin{center}
\begin{tabular}{l|lllll}
. & $a$ & $b$ & $c$ & $d$ & $e$ \\ \hline
$a$ & $a$ & $a$ & $a$ & $a$ & $a$ \\ 
$b$ & $a$ & $b$ & $c$ & $d$ & $e$ \\ 
$c$ & $a$ & $e$ & $b$ & $c$ & $d$ \\ 
$d$ & $a$ & $d$ & $e$ & $b$ & $c$ \\ 
$e$ & $a$ & $c$ & $d$ & $e$ & $b$%
\end{tabular}
\end{center}

Clearly $S$ is an intra-regular because, $a=(a\beta a^{2})\gamma a,$ $%
b=(c\beta b^{2})\gamma e,$ $c=(d\beta c^{2})\gamma e,$ $d=(c\beta
d^{2})\gamma c,$ $e=(b\beta e^{2})\gamma e.$

Note that in a $\Gamma $-AG-groupoid $S$ with left identity, $S=S\Gamma S.$

\begin{theorem}
\label{ji}A $\Gamma $-AG$^{\ast \ast }$-groupoid $S$ is an intra-regular $%
\Gamma $-AG$^{\ast \ast }$-groupoid if $S\Gamma a=S$ or $a\Gamma S=S$ holds
for all $a$ $\in S$.
\end{theorem}

\begin{proof}
Let $S$ be a $\Gamma $-AG$^{\ast \ast }$-groupoid such that $S\Gamma a=S$
holds for all $a\in S,$ then $S=S\Gamma S$. Let $a\in S$ and therefore, by
using $(2),$ we have%
\begin{eqnarray*}
a &\in &S=(S\Gamma S)\Gamma S=((S\Gamma a)\Gamma (S\Gamma a))\Gamma
S=((S\Gamma S)\Gamma (a\Gamma a))\Gamma S \\
&\subseteq &(S\Gamma a^{2})\Gamma S.
\end{eqnarray*}

Which shows that $S$ is an intra-regular $\Gamma $-AG$^{\ast \ast }$%
-groupoid.

Let $a\in S$ and assume that $a\Gamma S=S$ holds for all $a\in S,$ then by
using $(1)$, we have%
\begin{equation*}
a\in S=S\Gamma S=(a\Gamma S)\Gamma S=(S\Gamma S)\Gamma a=S\Gamma a.
\end{equation*}%
Thus $S\Gamma a=S$ holds for all $a$ $\in S$ and therefore it follows from
above that $S$ is an intra-regular.
\end{proof}

\begin{corollary}
If $S$ is a $\Gamma $-AG$^{\ast \ast }$-groupoid such that $a\Gamma S=S$
holds for all $a$ $\in S,$ then $S\Gamma a=S$ holds for all $a$ $\in S.$
\end{corollary}

\begin{theorem}
\label{ki}If $S$ is an intra-regular $\Gamma $-AG$^{\ast \ast }$-groupoid,
then $(B\Gamma S)\Gamma B=B\cap S,$ where $B$ is a $\Gamma $-bi-$(\Gamma $%
-generalized bi-$)$ ideal of $S$.
\end{theorem}

\begin{proof}
Let $S$ be an intra-regular $\Gamma $-AG$^{\ast \ast }$-groupoid, then
clearly $(B\Gamma S)\Gamma B\subseteq B\cap S$. Now let $b\in $ $B\cap S$
which implies that $b\in B$ and $b\in S,$ then since $S$ is an intra-regular 
$\Gamma $-AG$^{\ast \ast }$-groupoid so there exists $x,y\in S$ and $\alpha
,\beta ,\gamma \in $ $\Gamma $ such that $b=(x\alpha (b\beta b))\gamma y.$
Now by using $(3),$ $(1),$ $(4)$ and $(2),$ we have%
\begin{eqnarray*}
b &=&(x\alpha (b\beta b))\gamma y=(b\alpha (x\beta b))\gamma y=(y\alpha
(x\beta b))\gamma b \\
&=&(y\alpha (x\beta ((x\alpha (b\beta b))\gamma y)))\gamma b=(y\alpha
((x\alpha (b\beta b))\beta (x\gamma y)))\gamma b \\
&=&((x\alpha (b\beta b))\alpha (y\beta (x\gamma y)))\gamma b=(((x\gamma
y)\alpha y)\alpha ((b\beta b)\beta x))\gamma b \\
&=&((b\beta b)\alpha (((x\gamma y)\alpha y)\beta x))\gamma b=((b\beta
b)\alpha ((x\alpha y)\beta (x\gamma y)))\gamma b \\
&=&((b\beta b)\alpha ((x\alpha x)\beta (y\gamma y)))\gamma b=(((y\gamma
y)\beta (x\alpha x))\alpha (b\beta b))\gamma b \\
&=&(b\alpha (((y\gamma y)\beta (x\alpha x))\beta b))\gamma b\in (B\Gamma
S)\Gamma B.
\end{eqnarray*}

Which shows that $(B\Gamma S)\Gamma B=B\cap S.$
\end{proof}

\begin{corollary}
If $S$ is an intra-regular $\Gamma $-AG$^{\ast \ast }$-groupoid, then $%
(B\Gamma S)\Gamma B=B,$ where $B$ is a $\Gamma $-bi-$(\Gamma $-generalized
bi-$)$ ideal of $S$.
\end{corollary}

\begin{theorem}
\label{aw}If $S$ is an intra-regular $\Gamma $-AG$^{\ast \ast }$-groupoid,
then $(S\Gamma B)\Gamma S=S\cap B,$ where $B$ is a $\Gamma $-interior ideal
of $S$.
\end{theorem}

\begin{proof}
Let $S$ be an intra-regular $\Gamma $-AG$^{\ast \ast }$-groupoid, then
clearly $(S\Gamma B)\Gamma S\subseteq S\cap B$. Now let $b\in S\cap B$ which
implies that $b\in S$ and $b\in B,$ then since $S$ is an intra- regular $%
\Gamma $-AG$^{\ast \ast }$-groupoid so there exists $x,y\in S$ and $\alpha
,\gamma ,\delta \in $ $\Gamma $ such that $b=(x\alpha (b\delta b))\gamma y.$
Now by using $(3),$ $(1)$ and $(4),$ we have%
\begin{eqnarray*}
b &=&(x\alpha (b\delta b))\gamma y=(b\alpha (x\delta b))\gamma y=(y\alpha
(x\delta b))\gamma b \\
&=&(y\alpha (x\delta b))\gamma ((x\alpha (b\delta b))\gamma y)=(((x\alpha
(b\delta b))\gamma y)\alpha (x\delta b))\gamma y \\
&=&((b\gamma x)\alpha (y\delta (x\alpha (b\delta b))))\gamma y=(((y\delta
(x\alpha (b\delta b)))\gamma x)\alpha b)\gamma y\in (S\Gamma B)\Gamma S.
\end{eqnarray*}

Which shows that $(S\Gamma B)\Gamma S=S\cap B.$
\end{proof}

\begin{corollary}
If $S$ is an intra-regular $\Gamma $-AG$^{\ast \ast }$-groupoid, then $%
(S\Gamma B)\Gamma S=B,$ where $B$ is a $\Gamma $-interior ideal of $S$.
\end{corollary}

\begin{lemma}
\label{jk}If $S$ is an intra-regular regular $\Gamma $-AG$^{\ast \ast }$%
-groupoid, then $S=S\Gamma S.$
\end{lemma}

\begin{proof}
It is simple.
\end{proof}

\begin{lemma}
\label{LisR}A subset $A$ of an intra-regular $\Gamma $-AG$^{\ast \ast }$%
-groupoid $S$ is a left $\Gamma $-ideal if and only if it is a right $\Gamma 
$-ideal of $S.$
\end{lemma}

\begin{proof}
Let $S$ be an intra-regular $\Gamma $-AG$^{\ast \ast }$-groupoid and let $A$
be a right $\Gamma $-ideal of $S,$ then $A\Gamma S\subseteq A.$ Let $a\in A$
and since $S$ is an intra-regular $\Gamma $-AG$^{\ast \ast }$-groupoid so
there exists $x,y\in S$ and $\beta ,\gamma ,\delta \in $ $\Gamma $ such that 
$a=(x\beta (a\delta a))\gamma y.$ Let $p\in S\Gamma A$ and $\delta \in $ $%
\Gamma ,$ then by using $(3),$ $(1)$ and $(4),$ we have%
\begin{eqnarray*}
p &=&s\psi a=s\psi ((x\beta (a\delta a))\gamma y)=(x\beta (a\delta a))\psi
(s\gamma y)=(a\beta (x\delta a))\psi (s\gamma y) \\
&=&((s\gamma y)\beta (x\delta a))\psi a=((a\gamma x)\beta (y\delta s))\psi
a=(((y\delta s)\gamma x)\beta a)\psi a \\
&=&(a\beta a)\psi ((y\delta s)\gamma x)=(x\beta (y\delta s))\psi (a\gamma
a)=a\psi ((x\beta (y\delta s))\gamma a)\in A\Gamma S\subseteq A.
\end{eqnarray*}

Which shows that $A$ is a left $\Gamma $-ideal of $S.$

Let $A$ be a left $\Gamma $-ideal of $S,$ then $S\Gamma A\subseteq A.$ Let $%
a\in A$ and since $S$ is an intra-regular $\Gamma $-AG$^{\ast \ast }$%
-groupoid so there exists $x,y\in S$ and $\beta ,\gamma ,\delta \in $ $%
\Gamma $ such that $a=(x\beta (a\delta a))\gamma y.$ Let $p\in A\Gamma S$
and $\delta \in $ $\Gamma ,$ then by using $(1)$ and $(4),$ we have%
\begin{eqnarray*}
p &=&a\psi s=((x\beta (a\delta a)\gamma y)\psi s=(s\gamma y)\psi (x\beta
(a\delta a))=((a\delta a)\gamma x)\psi (y\beta s) \\
&=&((y\beta s)\gamma x)\psi (a\delta a)=(a\gamma a)\psi (x\delta (y\beta
s))=((x\delta (y\beta s))\gamma a)\psi a\in S\Gamma A\subseteq A.
\end{eqnarray*}

Which shows that $A$ is a right $\Gamma $-ideal of $S.$
\end{proof}

\begin{theorem}
\label{biiid}In an intra-regular $\Gamma $-AG$^{\ast \ast }$-groupoid $S$,
the following conditions are equivalent.
\end{theorem}

$(i)$ $A$ is a $\Gamma $-bi-($\Gamma $-generalized bi-) ideal of $S$.

$(ii)$ $(A\Gamma S)\Gamma A=A$ and $A\Gamma A=A.$

\begin{proof}
$(i)\Longrightarrow (ii):$ Let $A$ be a $\Gamma $-bi-ideal of an
intra-regular $\Gamma $-AG$^{\ast \ast }$-groupoid $S,$ then $(A\Gamma
S)\Gamma A\subseteq A$. Let $a\in A$, then since $S$ is an intra-regular so
there exists $x,$ $y\in S$ and $\beta ,\gamma ,\delta \in \Gamma $ such that 
$a=(x\beta (a\delta a))\gamma y.$ Now by using $(3),$ $(1),$ $(2)$ and $(4),$
we have%
\begin{eqnarray*}
a &=&(x\beta (a\delta a))\gamma y=(a\beta (x\delta a))\gamma y=(y\beta
(x\delta a))\gamma a \\
&=&(y\beta (x\delta ((x\beta (a\delta a))\gamma y)))\gamma a=(y\beta
((x\beta (a\delta a))\delta (x\gamma y)))\gamma a \\
&=&((x\beta (a\delta a))\beta (y\delta (x\gamma y)))\gamma a=((a\beta
(x\delta a))\beta (y\delta (x\gamma y)))\gamma a \\
&=&((a\beta y)\beta ((x\delta a)\delta (x\gamma y)))\gamma a=((x\delta
a)\beta ((a\beta y)\delta (x\gamma y)))\gamma a \\
&=&((x\delta a)\beta ((a\beta x)\delta (y\gamma y)))\gamma a=(((y\gamma
y)\delta (a\beta x))\beta (a\delta x))\gamma a \\
&=&(a\beta (((y\gamma y)\delta (a\beta x))\delta x))\gamma a\in (A\Gamma
S)\Gamma A.
\end{eqnarray*}

Thus $(A\Gamma S)\Gamma A=A$ holds. Now by using $(3),$ $(1),$ $(4)$ and $%
(2),$ we have%
\begin{eqnarray*}
a &=&(x\beta (a\delta a))\gamma y=(a\beta (x\delta a))\gamma y=(y\beta
(x\delta a))\gamma a \\
&=&(y\beta (x\delta ((x\beta (a\delta a))\gamma y)))\gamma a=(y\beta
((x\beta (a\delta a))\delta (x\gamma y)))\gamma a \\
&=&((x\beta (a\delta a))\beta (y\delta (x\gamma y)))\gamma a=((a\beta
(x\delta a))\beta (y\delta (x\gamma y)))\gamma a \\
&=&(((y\delta (x\gamma y))\beta (x\delta a))\beta a)\gamma a=(((a\delta
x)\beta ((x\gamma y)\delta y))\beta a)\gamma a \\
&=&(((a\delta x)\beta ((y\gamma y)\delta x))\beta a)\gamma a=(((a\delta
(y\gamma y))\beta (x\delta x))\beta a)\gamma a \\
&=&((((x\delta x)\delta (y\gamma y))\beta a)\beta a)\gamma a=((((x\delta
x)\delta (y\gamma y))\beta ((x\beta (a\delta a))\gamma y))\beta a)\gamma a \\
&=&((((x\delta x)\delta (y\gamma y))\beta ((a\beta (x\delta a))\gamma
y))\beta a)\gamma a \\
&=&((((x\delta x)\delta (a\beta (x\delta a)))\beta ((y\gamma y)\gamma
y))\beta a)\gamma a \\
&=&(((a\delta ((x\delta x)\beta (x\delta a)))\beta ((y\gamma y)\gamma
y))\beta a)\gamma a \\
&=&(((a\delta ((a\delta x)\beta (x\delta x)))\beta ((y\gamma y)\gamma
y))\beta a)\gamma a \\
&=&((((a\delta x)\delta (a\beta (x\delta x)))\beta ((y\gamma y)\gamma
y))\beta a)\gamma a \\
&=&((((a\delta a)\delta (x\beta (x\delta x)))\beta ((y\gamma y)\gamma
y))\beta a)\gamma a \\
&=&(((((y\gamma y)\gamma y)\delta (x\beta (x\delta x)))\beta (a\delta
a))\beta a)\gamma a \\
&=&((a\beta ((((y\gamma y)\gamma y)\delta (x\beta (x\delta x)))\delta
a))\beta a)\gamma a\subseteq ((A\Gamma S)\Gamma A)\Gamma A\subseteq A\Gamma
A.
\end{eqnarray*}

Hence $A=A\Gamma A$ holds.

$(ii)\Longrightarrow (i)$ is obvious.
\end{proof}

\begin{theorem}
In an intra-regular $\Gamma $-AG$^{\ast \ast }$-groupoid $S$, the following
conditions are equivalent.
\end{theorem}

$(i)$ $A$ is a $\Gamma $-$(1,2)$-ideal of $S$.

$(ii)$ $(A\Gamma S)\Gamma A^{2}=A$ and $A\Gamma A=A.$

\begin{proof}
$(i)$ $\Longrightarrow (ii):$ Let $A$ be a $\Gamma $-$(1,2)$-ideal of an
intra-regular $\Gamma $-AG$^{\ast \ast }$-groupoid $S,$ then $(A\Gamma
S)\Gamma A^{2}\subseteq A$ and $A\Gamma A\subseteq A$. Let $a\in A$, then
since $S$ is an intra- regular so there exists $x,$ $y\in S$ and $\beta
,\gamma ,\delta \in \Gamma $ such that $a=(x\beta (a\delta a)\gamma y.$ Now
by using $(3)$, $(1)$ and $(4),$ we have%
\begin{eqnarray*}
a &=&(x\beta (a\delta a))\gamma y=(a\beta (x\delta a))\gamma y=(y\beta
(x\delta a))\gamma a \\
&=&(y\beta (x\delta ((x\beta (a\delta a))\gamma y)))\gamma a=(y\beta
((x\beta (a\delta a))\delta (x\gamma y)))\gamma a \\
&=&((x\beta (a\delta a))\beta (y\delta (x\gamma y)))\gamma a=(((x\gamma
y)\beta y)\beta ((a\delta a)\delta x))\gamma a \\
&=&(((y\gamma y)\beta x)\beta ((a\delta a)\delta x))\gamma a=((a\delta
a)\beta (((y\gamma y)\beta x)\delta x))\gamma a \\
&=&((a\delta a)\beta ((x\beta x)\delta (y\gamma y)))\gamma a=(a\beta
((x\beta x)\delta (y\gamma y)))\gamma (a\delta a)\in (A\Gamma S)\Gamma
A\Gamma A.
\end{eqnarray*}

Thus $(A\Gamma S)\Gamma A^{2}=A.$ Now by using $(3),$ $(1),$ $(4)$ and $(2),$
we have%
\begin{eqnarray*}
a &=&(x\beta (a\delta a))\gamma y=(a\beta (x\delta a))\gamma y=(y\beta
(x\delta a))\gamma a \\
&=&(y\beta (x\delta a))\gamma ((x\beta (a\delta a))\gamma y)=(x\beta
(a\delta a))\gamma ((y\beta (x\delta a))\gamma y) \\
&=&(a\beta (x\delta a))\gamma ((y\beta (x\delta a))\gamma y)=(((y\beta
(x\delta a))\gamma y)\beta (x\delta a))\gamma a \\
&=&((a\gamma x)\beta (y\delta (y\beta (x\delta a))))\gamma a \\
&=&((((x\beta (a\delta a))\gamma y)\gamma x)\beta (y\delta (y\beta (x\delta
a))))\gamma a \\
&=&(((x\gamma y)\gamma (x\beta (a\delta a)))\beta (y\delta (y\beta (x\delta
a))))\gamma a \\
&=&(((x\gamma y)\gamma y)\beta ((x\beta (a\delta a))\delta (y\beta (x\delta
a))))\gamma a \\
&=&(((y\gamma y)\gamma x)\beta ((x\beta (a\delta a))\delta (y\beta (x\delta
a))))\gamma a \\
&=&(((y\gamma y)\gamma x)\beta ((x\beta y)\delta ((a\delta a)\beta (x\delta
a))))\gamma a \\
&=&(((y\gamma y)\gamma x)\beta ((a\delta a)\delta ((x\beta y)\beta (x\delta
a))))\gamma a \\
&=&((a\delta a)\beta (((y\gamma y)\gamma x)\delta ((x\beta y)\beta (x\delta
a))))\gamma a \\
&=&((a\delta a)\beta (((y\gamma y)\gamma x)\delta ((x\beta x)\beta (y\delta
a))))\gamma a \\
&=&((((x\beta x)\beta (y\delta a))\delta ((y\gamma y)\gamma x))\beta
(a\delta a))\gamma a \\
&=&((((a\beta y)\beta (x\delta x))\delta ((y\gamma y)\gamma x))\beta
(a\delta a))\gamma a \\
&=&(((((x\delta x)\beta y)\beta a)\delta ((y\gamma y)\gamma x))\beta
(a\delta a))\gamma a \\
&=&(((x\beta (y\gamma y))\delta (a\gamma ((x\delta x)\beta y)))\beta
(a\delta a))\gamma a \\
&=&((a\delta ((x\beta (y\gamma y))\gamma ((x\delta x)\beta y)))\beta
(a\delta a))\gamma a \\
&=&((a\delta ((x\beta (x\delta x))\gamma ((y\gamma y)\beta y)))\beta
(a\delta a))\gamma a \\
&\in &((A\Gamma S)\Gamma A^{2})\Gamma A\subseteq A\Gamma A.
\end{eqnarray*}

Hence $A\Gamma A=A.$

$(ii)$ $\Longrightarrow (i)$ is obvious.
\end{proof}

\begin{theorem}
In an intra-regular $\Gamma $-AG$^{\ast \ast }$-groupoid $S$, the following
conditions are equivalent.
\end{theorem}

$(i)$ $A$ is a $\Gamma $-interior ideal of $S$.

$(ii)$ $(S\Gamma A)\Gamma S=A.$

\begin{proof}
$(i)$ $\Longrightarrow (ii):$ Let $A$ be a $\Gamma $-interior ideal of an
intra-regular $\Gamma $-AG$^{\ast \ast }$-groupoid $S,$ then $(S\Gamma
A)\Gamma S\subseteq A$. Let $a\in A$, then since $S$ is an intra- regular so
there exists $x,$ $y\in S$ and $\beta ,\gamma ,\delta \in \Gamma $ such that 
$a=(x\beta (a\delta a))\gamma y.$ Now by using $(3)$, $(1)$ and $(4),$ we
have%
\begin{eqnarray*}
a &=&(x\beta (a\delta a))\gamma y=(a\beta (x\delta a))\gamma y=(y\beta
(x\delta a))\gamma a \\
&=&(y\beta (x\delta a))\gamma ((x\beta (a\delta a))\gamma y)=(((x\beta
(a\delta a))\gamma y)\beta (x\delta a))\gamma y \\
&=&((a\gamma x)\beta (y\delta (x\beta (a\delta a))))\gamma y=(((y\delta
(x\beta (a\delta a)))\gamma x)\beta a)\delta y\in (S\Gamma A)\Gamma S.
\end{eqnarray*}

Thus $(S\Gamma A)\Gamma S=A.$

$(ii)$ $\Longrightarrow (i)$ is obvious.
\end{proof}

\begin{theorem}
In an intra-regular $\Gamma $-AG$^{\ast \ast }$-groupoid $S$, the following
conditions are equivalent.
\end{theorem}

$(i)$ $A$ is a $\Gamma $-quasi ideal of $S$.

$(ii)$ $S\Gamma Q\cap Q\Gamma S=Q.$

\begin{proof}
$(i)\Longrightarrow (ii):$ Let $Q$ be a $\Gamma $-quasi ideal of an
intra-regular $\Gamma $-AG$^{\ast \ast }$-groupoid $S,$ then $S\Gamma Q\cap
Q\Gamma S\subseteq Q$. Let $q\in Q$, then since $S$ is an intra- regular so
there exists $x$, $y\in S$ and $\alpha ,\beta ,\gamma \in \Gamma $ such that 
$q=(x\alpha (q\gamma q))\beta y.$ Let $p\delta q\in S\Gamma Q,$ for some $%
\delta \in $ $\Gamma ,$ then by using $(3)$, $(2)$ and $(4),$ we have%
\begin{eqnarray*}
p\delta q &=&p\delta ((x\alpha (q\gamma q))\beta y)=(x\alpha (q\gamma
q))\delta (p\beta y)=(q\alpha (x\gamma q))\delta (p\beta y) \\
&=&(q\alpha p)\delta ((x\gamma q)\beta y)=(x\gamma q)\delta ((q\alpha
p)\beta y)=(y\gamma (q\alpha p))\delta (q\beta x) \\
&=&q\delta ((y\gamma (q\alpha p))\beta x)\in Q\Gamma S.
\end{eqnarray*}

Now let $q\delta y\in Q\Gamma S,$ then by using $(1)$, $(3)$ and $(4),$ we
have%
\begin{eqnarray*}
q\delta p &=&((x\alpha (q\gamma q))\beta y)\delta p=(p\beta y)\delta
(x\alpha (q\gamma q))=x\delta ((p\beta y)\alpha (q\gamma q)) \\
&=&x\delta ((q\beta q)\alpha (y\gamma p))=(q\beta q)\delta (x\alpha (y\gamma
p))=((x\alpha (y\gamma p))\beta q)\delta q\in S\Gamma Q.
\end{eqnarray*}

Hence $Q\Gamma S=S\Gamma Q.$ As by using $(3)$ and $(1),$ we have%
\begin{equation*}
q=(x\alpha (q\gamma q))\beta y=(q\alpha (x\gamma q))\beta y=(y\alpha
(x\gamma q))\beta q\in S\Gamma Q.
\end{equation*}%
Thus $q\in S\Gamma Q\cap Q\Gamma S$ implies that $S\Gamma Q\cap Q\Gamma S=Q$.

$(ii)\Longrightarrow (i)$ is obvious.
\end{proof}

\begin{theorem}
\label{12}In an intra-regular $\Gamma $-AG$^{\ast \ast }$-groupoid $S$, the
following conditions are equivalent.
\end{theorem}

$(i)$ $A$ is a $\Gamma $-$(1,2)$-ideal of $S$.

$(ii)$ $A$ is a two-sided $\Gamma $-ideal of $S.$

\begin{proof}
$(i)$ $\Longrightarrow (ii):$ Let $S$ be an intra-regular $\Gamma $-AG$%
^{\ast \ast }$-groupoid and let $A$ be a $\Gamma $-$(1,2)$-ideal of $S,$
then $(A\Gamma S)\Gamma A^{2}\subseteq A.$ Let $a\in A$, then since $S$ is
an intra-regular so there exists $x,$ $y\in S$ and $\beta ,\gamma ,\delta
\in \Gamma ,$ such that $a=(x\beta (a\delta a))\gamma y.$ Now let $\psi \in
\Gamma ,$ then by using $(3),$ $(1)$ and $(4),$ we have%
\begin{eqnarray*}
s\psi a &=&s\psi ((x\beta (a\delta a))\gamma y)=(x\beta (a\delta a))\psi
(s\gamma y)=(a\beta (x\delta a))\psi (s\gamma y) \\
&=&((s\gamma y)\beta (x\delta a))\psi a=((s\gamma y)\beta (x\delta a))\psi
((x\beta (a\delta a))\gamma y) \\
&=&(x\beta (a\delta a))\psi (((s\gamma y)\beta (x\delta a))\gamma y)=(y\beta
((s\gamma y)\beta (x\delta a)))\psi ((a\delta a)\gamma x) \\
&=&(a\delta a)\psi ((y\beta ((s\gamma y)\beta (x\delta a)))\gamma
x)=(x\delta (y\beta ((s\gamma y)\beta (x\delta a))))\psi (a\gamma a) \\
&=&(x\delta (y\beta ((a\gamma x)\beta (y\delta s))))\psi (a\gamma
a)=(x\delta ((a\gamma x)\beta (y\beta (y\delta s))))\psi (a\gamma a) \\
&=&((a\gamma x)\delta (x\beta (y\beta (y\delta s))))\psi (a\gamma a) \\
&=&((((x\beta (a\delta a))\gamma y)\gamma x)\delta (x\beta (y\beta (y\delta
s))))\psi (a\gamma a) \\
&=&(((x\gamma y)\gamma (x\beta (a\delta a)))\delta (x\beta (y\beta (y\delta
s))))\psi (a\gamma a) \\
&=&((((a\delta a)\gamma x)\gamma (y\beta x))\delta (x\beta (y\beta (y\delta
s))))\psi (a\gamma a) \\
&=&((((y\beta x)\gamma x)\gamma (a\delta a))\delta (x\beta (y\beta (y\delta
s))))\psi (a\gamma a) \\
&=&(((y\beta (y\delta s))\gamma x)\delta ((a\delta a)\beta ((y\beta x)\gamma
x)))\psi (a\gamma a) \\
&=&(((y\beta (y\delta s))\gamma x)\delta ((a\delta a)\beta ((x\beta x)\gamma
y)))\psi (a\gamma a) \\
&=&((a\delta a)\delta (((y\beta (y\delta s))\gamma x)\beta ((x\beta x)\gamma
y)))\psi (a\gamma a) \\
&=&((((x\beta x)\gamma y)\delta ((y\beta (y\delta s))\gamma x))\delta
(a\beta a))\psi (a\gamma a) \\
&=&(a\delta (((x\beta x)\gamma y)\delta (((y\beta (y\delta s))\gamma x)\beta
a)))\psi (a\gamma a)\in (A\Gamma S)\Gamma A^{2}\subseteq A.
\end{eqnarray*}

Hence $A$ is a left $\Gamma $-ideal of $S$ and by Lemma \ref{LisR}, $A$ is a
two-sided $\Gamma $-ideal of $S.$

$(ii)$ $\Longrightarrow (i):$ Let $A$ be a two-sided $\Gamma $-ideal of $S$.
Let $y\in (A\Gamma S)\Gamma A^{2},$ then $y=(a\beta s)\gamma (b\delta b)$
for some $a,b\in A$, $s\in S$ and $\beta ,\gamma ,\delta \in \Gamma .$ Now
by using $(3),$ we have%
\begin{equation*}
y=(a\beta s)\gamma (b\delta b)=b\gamma ((a\beta s)\delta b)\in A\Gamma
S\subseteq A.
\end{equation*}

Hence $(A\Gamma S)\Gamma A^{2}\subseteq A$ and therefore $A$ is a $\Gamma $-$%
(1,2)$-ideal of $S$.
\end{proof}

\begin{theorem}
\label{plo}In an intra-regular $\Gamma $-AG$^{\ast \ast }$-groupoid $S$, the
following conditions are equivalent.
\end{theorem}

$(i)$ $A$ is a $\Gamma $-$(1,2)$-ideal of $S$.

$(ii)$ $A$ is a $\Gamma $-interior ideal of $S.$

\begin{proof}
$(i)$ $\implies (ii):$ Let $A$ be a $\Gamma $-$(1,2)$-ideal of an
intra-regular $\Gamma $-AG$^{\ast \ast }$-groupoid $S,$ then $(A\Gamma
S)\Gamma A^{2}\subseteq A.$ Let $p\in (S\Gamma A)\Gamma S,$ then $p=(s\mu
a)\psi s^{^{\prime }}$ for some $a\in A$, $s,s^{^{\prime }}\in S$ and $\mu
,\psi \in \Gamma $. Since $S$ is intra-regular so there exists $x,$ $y\in S$
and $\beta ,\gamma ,\delta \in \Gamma $ such that $a=(x\beta (a\delta
a))\gamma y.$ Now by using $(3),$ $(1)$, $(2)$ and $(4),$ we have%
\begin{eqnarray*}
p &=&(s\mu a)\psi s^{^{\prime }}=(s\mu ((x\beta (a\delta a))\gamma y))\psi
s^{^{\prime }}=((x\beta (a\delta a))\mu (s\gamma y))\psi s^{^{\prime }} \\
&=&(s^{^{\prime }}\mu (s\gamma y))\psi (x\beta (a\delta a))=(s^{^{\prime
}}\mu (s\gamma y))\psi (a\beta (x\delta a)) \\
&=&a\psi ((s^{^{\prime }}\mu (s\gamma y))\beta (x\delta a))=((x\beta
(a\delta a))\gamma y)\psi ((s^{^{\prime }}\mu (s\gamma y))\beta (x\delta a))
\\
&=&((a\beta (x\delta a))\gamma y)\psi ((s^{^{\prime }}\mu (s\gamma y))\beta
(x\delta a)) \\
&=&((a\beta (x\delta a))\gamma (s^{^{\prime }}\mu (s\gamma y)))\psi (y\beta
(x\delta a)) \\
&=&((a\beta s^{^{\prime }})\gamma ((x\delta a)\mu (s\gamma y)))\psi (y\beta
(x\delta a)) \\
&=&((a\beta s^{^{\prime }})\gamma ((y\delta s)\mu (a\gamma x)))\psi (y\beta
(x\delta a)) \\
&=&((a\beta s^{^{\prime }})\gamma (a\mu ((y\delta s)\gamma x)))\psi (y\beta
(x\delta a)) \\
&=&((a\beta a)\gamma (s^{^{\prime }}\mu ((y\delta s)\gamma x)))\psi (y\beta
(x\delta a)) \\
&=&((a\beta a)\gamma ((y\delta s)\mu (s^{^{\prime }}\gamma x)))\psi (y\beta
(x\delta a)) \\
&=&((y\beta (x\delta a))\gamma ((y\delta s)\mu (s^{^{\prime }}\gamma
x)))\psi (a\beta a) \\
&=&((y\beta (y\delta s))\gamma ((x\delta a)\mu (s^{^{\prime }}\gamma
x)))\psi (a\beta a) \\
&=&((y\beta (y\delta s))\gamma ((x\delta s^{^{\prime }})\mu (a\gamma
x)))\psi (a\beta a) \\
&=&((y\beta (y\delta s))\gamma (a\mu ((x\delta s^{^{\prime }})\gamma
x)))\psi (a\beta a) \\
&=&(a\gamma ((y\beta (y\delta s))\mu ((x\delta s^{^{\prime }})\gamma
x)))\psi (a\beta a) \\
&\in &(A\Gamma S)\Gamma A^{2}\subseteq A.
\end{eqnarray*}%
Thus $(S\Gamma A)\Gamma S\subseteq A.$ Which shows that $A$ is a $\Gamma $%
-interior ideal of $S.$

$(ii)$ $\implies (i):$ Let $A$ be a $\Gamma $-interior ideal of $S,$ then $%
(S\Gamma A)\Gamma S\subseteq A.$ Let $p\in (A\Gamma S)\Gamma A^{2},$ then $%
p=(a\mu s)\psi (b\alpha b),$ for some $a,b\in A$, $s\in S$ and $\mu ,\psi
,\alpha \in \Gamma $. Since $S$ is intra-regular so there exists $x,$ $y\in
S $ and $\beta ,\gamma ,\delta \in \Gamma $ such that $a=(x\beta (a\delta
a))\gamma y.$ Now by using $(1),$ $(3)$ and $(4),$ we have%
\begin{eqnarray*}
p &=&(a\mu s)\psi (b\alpha b)=((b\alpha b)\mu s)\psi a=((b\alpha b)\mu
s)\psi ((x\beta (a\gamma a))\gamma y) \\
&=&(x\beta (a\gamma a))\psi (((b\alpha b)\mu s)\gamma y)=((((b\alpha b)\mu
s)\gamma y)\beta (a\gamma a))\psi x \\
&=&((a\gamma a)\beta (y\delta ((b\alpha b)\mu s)))\psi x=(((y\delta
((b\alpha b)\mu s))\gamma a)\beta a)\psi x\in (S\Gamma A)\Gamma S\subseteq A.
\end{eqnarray*}%
Thus $(A\Gamma S)\Gamma A^{2}\subseteq A$.

Now by using $(3)$ and $(4),$ we have%
\begin{eqnarray*}
A\Gamma A &\subseteq &A\Gamma S=A\Gamma (S\Gamma S)=S\Gamma (A\Gamma
S)=(S\Gamma S)\Gamma (A\Gamma S) \\
&=&(S\Gamma A)\Gamma (S\Gamma S)=(S\Gamma A)\Gamma S\subseteq A.
\end{eqnarray*}

Which shows that $A$ is a $\Gamma $-$(1,2)$-ideal of $S.$
\end{proof}

\begin{theorem}
\label{bint}In an intra-regular $\Gamma $-AG$^{\ast \ast }$-groupoid $S$,
the following conditions are equivalent.
\end{theorem}

$(i)$ $A$ is a $\Gamma $-bi-ideal of $S.$

$(ii)$ $A$ is a $\Gamma $-interior ideal of $S.$

\begin{proof}
$(i)$ $\implies (ii):$ Let $A$ be a $\Gamma $-bi-ideal of an intra-regular $%
\Gamma $-AG$^{\ast \ast }$-groupoid $S,$ then $(A\Gamma S)\Gamma A\subseteq
A.$ Let $p\in (S\Gamma A)\Gamma S,$ then $p=(s\mu a)\psi s^{^{\prime }}$ for
some $a\in A$, $s,s^{^{\prime }}\in S$ and $\mu ,\psi \in \Gamma $. Since $S$
is an intra-regular so there exists $x,$ $y\in S$ and $\beta ,\gamma ,\delta
\in \Gamma $ such that $a=(x\beta (a\delta a))\gamma y.$ Now by using $(3),$ 
$(1)$, $(4)$ and $(2),$ we have%
\begin{eqnarray*}
p &=&(s\mu a)\psi s^{^{\prime }}=(s\mu ((x\beta (a\delta a))\gamma y))\psi
s^{^{\prime }}=((x\beta (a\delta a))\mu (s\gamma y))\psi s^{^{\prime }} \\
&=&(s^{^{\prime }}\mu (s\gamma y))\psi (x\beta (a\delta a))=((a\delta a)\mu
x)\psi ((s\gamma y)\beta s^{^{\prime }}) \\
&=&(((s\gamma y)\beta s^{^{\prime }})\mu x)\psi (a\delta a)=((x\beta
s^{^{\prime }})\mu (s\gamma y))\psi (a\delta a) \\
&=&(a\mu a)\psi ((s\gamma y)\delta (x\beta s^{^{\prime }}))=(((s\gamma
y)\delta (x\beta s^{^{\prime }}))\mu a)\psi a \\
&=&(((s\gamma y)\delta (x\beta s^{^{\prime }}))\mu ((x\beta (a\delta
a))\gamma y))\psi a \\
&=&(((s\gamma y)\delta (x\beta (a\delta a)))\mu ((x\beta s^{^{\prime
}})\gamma y))\psi a \\
&=&((((a\delta a)\gamma x)\delta (y\beta s))\mu ((x\beta s^{^{\prime
}})\gamma y))\psi a \\
&=&((((x\beta s^{^{\prime }})\gamma y)\delta (y\beta s))\mu ((a\delta
a)\gamma x))\psi a \\
&=&((a\delta a)\mu ((((x\beta s^{^{\prime }})\gamma y)\delta (y\beta
s))\gamma x))\psi a \\
&=&((x\delta (((x\beta s^{^{\prime }})\gamma y)\delta (y\beta s)))\mu
(a\gamma a))\psi a \\
&=&(a\mu ((x\delta (((x\beta s^{^{\prime }})\gamma y)\delta (y\beta
s)))\gamma a))\psi a \\
&\in &(A\Gamma S)\Gamma A\subseteq A.
\end{eqnarray*}

Thus $(S\Gamma A)\Gamma S\subseteq A.$ Which shows that $A$ is a $\Gamma $%
-interior ideal of $S.$

$(ii)$ $\implies (i):$ Let $A$ be a $\Gamma $-interior ideal of $S,$ then $%
(S\Gamma A)\Gamma S\subseteq A.$ Let $p\in (A\Gamma S)\Gamma A,$ then $%
p=(a\mu s)\psi b$ for some $a,b\in A$, $s\in S$ and $\mu ,\psi \in \Gamma .$
Since $S$ is an intra-regular so there exists $x,$ $y\in S$ and $\beta
,\gamma ,\delta \in \Gamma $ such that $b=(x\beta (b\delta b))\gamma y.$ Now
by using $(3),$ $(1)$ and $(4),$ we have%
\begin{eqnarray*}
p &=&(a\mu s)\psi b=(a\mu s)\psi ((x\beta (b\delta b))\gamma y)=(x\beta
(b\delta b))\psi ((a\mu s)\gamma y) \\
&=&(((a\mu s)\gamma y)\beta (b\delta b))\psi x=((b\gamma b)\beta (y\delta
(a\mu s)))\psi x \\
&=&(((y\delta (a\mu s))\gamma b)\beta b)\psi x\in (S\Gamma A)\Gamma
S\subseteq A.
\end{eqnarray*}

Thus $(A\Gamma S)\Gamma A\subseteq A.$

Now 
\begin{eqnarray*}
A\Gamma A &\subseteq &A\Gamma S=A\Gamma (S\Gamma S)=S\Gamma (A\Gamma
S)=(S\Gamma S)\Gamma (A\Gamma S) \\
&=&(S\Gamma A)\Gamma (S\Gamma S)=(S\Gamma A)\Gamma S\subseteq A.
\end{eqnarray*}%
Which shows that $A$ is a $\Gamma $-bi-ideal of $S.$
\end{proof}

\begin{theorem}
\label{quo}In an intra-regular $\Gamma $-AG$^{\ast \ast }$-groupoid $S$, the
following conditions are equivalent.
\end{theorem}

$(i)$ $A$ is a $\Gamma $-$(1,2)$-ideal of $S$.

$(ii)$ $A$ is a $\Gamma $-quasi ideal of $S.$

\begin{proof}
$(i)$ $\implies (ii):$ Let $A$ be a $\Gamma $-$(1,2)$-ideal of intra-regular 
$\Gamma $-AG$^{\ast \ast }$-groupoid $S,$ then $(A\Gamma S)\Gamma (A\Gamma
A)\subseteq A.$ Now by using $(3)$ and $(4),$ we have 
\begin{equation*}
S\Gamma A=S\Gamma (A\Gamma A)=S\Gamma ((A\Gamma A)\Gamma A)=(A\Gamma
A)\Gamma (S\Gamma A)=(A\Gamma S)\Gamma (A\Gamma A)\subseteq A.
\end{equation*}

and by using $(1)$ and $(3),$ we have%
\begin{eqnarray*}
A\Gamma S &=&(A\Gamma A)\Gamma S=((A\Gamma A)\Gamma A)\Gamma S=(S\Gamma
A)\Gamma (A\Gamma A)=(S\Gamma (A\Gamma A))\Gamma (A\Gamma A) \\
&=&((S\Gamma S)\Gamma (A\Gamma A))\Gamma (A\Gamma A)=((A\Gamma A)\Gamma
(S\Gamma S))\Gamma (A\Gamma A)=(A\Gamma S)\Gamma (A\Gamma A)\subseteq A.
\end{eqnarray*}

Hence $(A\Gamma S)\cap (S\Gamma A)\subseteq A.$ Which shows that $A$ is a $%
\Gamma $-quasi ideal of $S.$

$(ii)$ $\implies (i):$ Let $A$ be a $\Gamma $-quasi ideal of $S,$ then $%
(A\Gamma S)\cap (S\Gamma A)\subseteq A.$ Now $A\Gamma A\subseteq A\Gamma S$
and $A\Gamma A\subseteq S\Gamma A.$ Thus $A\Gamma A\subseteq (A\Gamma S)\cap
(S\Gamma A)\subseteq A.$

Now by using $(4)$ and $(3),$ we have%
\begin{equation*}
(A\Gamma S)\Gamma (A\Gamma A)=(A\Gamma A)\Gamma (S\Gamma A)\subseteq A\Gamma
(S\Gamma A)=S\Gamma (A\Gamma A)\subseteq S\Gamma A.
\end{equation*}

and%
\begin{eqnarray*}
(A\Gamma S)\Gamma (A\Gamma A) &=&(A\Gamma A)\Gamma (S\Gamma A)\subseteq
A\Gamma (S\Gamma A)=S\Gamma (A\Gamma A) \\
&=&(S\Gamma S)\Gamma (A\Gamma A)=(A\Gamma A)\Gamma (S\Gamma S)\subseteq
A\Gamma S.
\end{eqnarray*}%
Thus $(A\Gamma S)\Gamma (A\Gamma A)\subseteq (A\Gamma S)\cap (S\Gamma
A)\subseteq A.$ Which shows that $A$ is a $\Gamma $-$(1,2)$-ideal of $S$.
\end{proof}

\begin{lemma}
\label{li}Let$\ A$ be a subset of an intra-regular $\Gamma $-AG$^{\ast \ast
} $-groupoid $S$, then $A$ is a two-sided $\Gamma $-ideal of $S$ if and only
if $A\Gamma S=A$ and $S\Gamma A=A$.
\end{lemma}

\begin{proof}
It is simple.
\end{proof}

\begin{theorem}
\label{equalient}For an intra-regular $\Gamma $-AG$^{\ast \ast }$-groupoid $%
S $ the following statements are equivalent.
\end{theorem}

$\left( i\right) $ $A$ is a left $\Gamma $-ideal of $S$.

$\left( ii\right) $ $A$ is a right $\Gamma $-ideal of $S$.

$\left( iii\right) $ $A$ is a two-sided $\Gamma $-ideal of $S$.

$\left( iv\right) $ $A\Gamma S=A$ and $S\Gamma A=A$.

$\left( v\right) $ $A$ is a $\Gamma $-quasi ideal of $S$.

$\left( vi\right) $ $A$ is a $\Gamma $-$(1,2)$-ideal of $S$.

$\left( vii\right) $ $A$ is a $\Gamma $-generalized bi-ideal of $S$.

$\left( viii\right) $ $A$ is a $\Gamma $-bi-ideal of $S$.

$(ix)$ $A$ is a $\Gamma $-interior ideal of $S$.

\begin{proof}
$\left( i\right) \Longrightarrow \left( ii\right) $ and $\left( ii\right)
\Longrightarrow \left( iii\right) $ are followed by Lemma \ref{LisR}.

$\left( iii\right) \Longrightarrow \left( iv\right) $ is followed by Lemma %
\ref{li}, and $\left( iv\right) \Longrightarrow \left( v\right) $ is obvious.

$\left( v\right) \Longrightarrow \left( vi\right) $ is followed by Theorem %
\ref{quo}.

$\left( vi\right) \Longrightarrow \left( vii\right) :$ Let $A$ be a $\Gamma $%
-$(1,2)$-ideal of an intra-regular $\Gamma $-AG$^{\ast \ast }$-groupoid $S$,
then $(A\Gamma S)\Gamma A^{2}\subseteq A$. Let $p\in (A\Gamma S)\Gamma A,$
then $p=(a\mu s)\psi b$ for some $a,b\in A$, $s\in S$ and $\mu ,\psi \in
\Gamma $. Now since $S$ is an intra-regular so there exists $x,$ $y\in S$
and $\beta ,\gamma ,\delta \in \Gamma $ such that such that $b=(x\beta
(b\delta b))\gamma y$ then, by using $(3)$ and $(4),$ we have%
\begin{eqnarray*}
p &=&(a\mu s)\psi b=(a\mu s)\psi ((x\beta (b\delta b))\gamma y)=(x\beta
(b\delta b))\psi ((a\mu s)\gamma y) \\
&=&(y\beta (a\mu s))\psi ((b\delta b)\gamma x)=(b\delta b)\psi ((y\beta
(a\mu s))\gamma x) \\
&=&(x\delta (y\beta (a\mu s)))\psi (b\gamma b)=(x\delta (a\beta (y\mu
s)))\psi (b\delta b) \\
&=&(a\delta (x\beta (y\mu s)))\psi (b\delta b)\in (A\Gamma S)\Gamma
A^{2}\subseteq A.
\end{eqnarray*}

Which shows that $A$ is a $\Gamma $-generalized bi-ideal of $S$.

$\left( vii\right) \Longrightarrow \left( viii\right) $ is simple.

$(viii)\Longrightarrow \left( ix\right) $ is followed by Theorem \ref{bint}.

$(ix)\Longrightarrow \left( i\right) $ is followed by Theorems \ref{plo} and %
\ref{12}.
\end{proof}

\begin{theorem}
\label{ii}In a $\Gamma $-AG$^{\ast \ast }$-groupoid $S$, the following
conditions are equivalent.
\end{theorem}

$(i)$ $S$ is intra-regular.

$(ii)$ Every $\Gamma $-bi-ideal of $S$ is $\Gamma $-idempotent.

\begin{proof}
$(i)\Longrightarrow (ii)$ is obvious by Theorem \ref{biiid}.

$(ii)\Longrightarrow (i):$ Since $S\Gamma a$ is a $\Gamma $-bi-ideal of $S$,
and by assumption $S\Gamma a$ is $\Gamma $-idempotent, so by using $(2)$, we
have%
\begin{eqnarray*}
a &\in &\left( S\Gamma a\right) \Gamma \left( S\Gamma a\right) =\left(
\left( S\Gamma a\right) \Gamma \left( S\Gamma a\right) \right) \Gamma \left(
S\Gamma a\right) \\
&=&\left( \left( S\Gamma S\right) \Gamma \left( a\Gamma a\right) \right)
\Gamma \left( S\Gamma a\right) \subseteq \left( S\Gamma a^{2}\right) \Gamma
\left( S\Gamma S\right) =\left( S\Gamma a^{2}\right) \Gamma S\text{.}
\end{eqnarray*}

Hence $S$ is intra-regular.
\end{proof}

\begin{lemma}
\label{idl} If $I$ and $J$ are two-sided $\Gamma $-ideals of an
intra-regular $\Gamma $-AG$^{\ast \ast }$-groupoid $S$ , then $I\cap J$ is a
two-sided $\Gamma $-ideal of $S.$
\end{lemma}

\begin{proof}
It is simple.
\end{proof}

\begin{lemma}
\label{ij}In an intra-regular $\Gamma $-AG$^{\ast \ast }$-groupoid $I\Gamma
J=I\cap J$, for every $\Gamma $-ideals $I$ and $J$ in $S$.
\end{lemma}

\begin{proof}
Let $I$ and $J$ be any $\Gamma $-ideals of $S$, then obviously $I\Gamma
J\subseteq I\cap J$. Since $I\cap J\subseteq I$ and $I\cap J\subseteq J$,
then $\left( I\cap J\right) ^{2}\subseteq I\Gamma J$, also by Lemma \ref{idl}%
, $I\cap J$ is a $\Gamma $-ideal of $S,$ so by Theorem $\ref{ii}$, we have $%
I\cap J=\left( I\cap J\right) ^{2}\subseteq I\Gamma J$. Hence $I\Gamma
J=I\cap J$.
\end{proof}

\begin{lemma}
\label{iffff}Let $S$ be a $\Gamma $-AG$^{\ast \ast }$-groupoid, then $S$ is
an intra-regular if and only if every left $\Gamma $-ideal of $S$ is $\Gamma 
$-idempotent.
\end{lemma}

\begin{proof}
Let $S$ be an intra-regular $\Gamma $-AG$^{\ast \ast }$-groupoid, then by
Theorems \ref{equalient} and $\ref{ii}$, every $\Gamma $-ideal of $S$ is $%
\Gamma $-idempotent.

Conversely, assume that every left $\Gamma $-ideal of $S$ is $\Gamma $%
-idempotent. Since $S\Gamma a$ is a left $\Gamma $-ideal of $S$, so by using 
$(2)$, we have%
\begin{eqnarray*}
a &\in &S\Gamma a=\left( S\Gamma a\right) \Gamma \left( S\Gamma a\right)
=\left( \left( S\Gamma a\right) \Gamma \left( S\Gamma a\right) \right)
\Gamma \left( S\Gamma a\right) \\
&=&\left( \left( S\Gamma S\right) \Gamma \left( a\Gamma a\right) \right)
\Gamma \left( S\Gamma a\right) \subseteq (S\Gamma a^{2})\Gamma (S\Gamma
S)=\left( S\Gamma a^{2}\right) \Gamma S\text{.}
\end{eqnarray*}

Hence $S$ is intra-regular.
\end{proof}

\begin{lemma}
In an AG$^{\ast \ast }$-groupoid $S$, the following conditions are
equivalent.
\end{lemma}

$(i)$ $S$ is intra-regular.

$(ii)$ $A=(S\Gamma A)^{2},$ where $A$ is any left $\Gamma $-ideal of S.

\begin{proof}
$(i)$ $\Longrightarrow (ii):$ Let $A$ be a left $\Gamma $-ideal of an
intra-regular $\Gamma $-AG$^{\ast \ast }$-groupoid $S,$ then $S\Gamma
A\subseteq A$ and by Lemma \ref{iffff}, $(S\Gamma A)^{2}=S\Gamma A\subseteq
A.$ Now $A=A\Gamma A\subseteq S\Gamma A=(S\Gamma A)^{2},$ which implies that 
$A=(S\Gamma A)^{2}.$

$(ii)$ $\Longrightarrow (i):$ Let $A$ be a left $\Gamma $-ideal of $S,$ then 
$A=(S\Gamma A)^{2}\subseteq A\Gamma A,$ which implies that $A$ is $\Gamma $%
-idempotent and by using Lemma \ref{iffff}, $S$ is an intra-regular.
\end{proof}

\begin{theorem}
\label{RLT}For an intra-regular $\Gamma $-AG$^{\ast \ast }$-groupoid $S$,
the following statements holds.
\end{theorem}

$\left( i\right) $ Every right $\Gamma $-ideal of $S$ is $\Gamma $-semiprime.

$\left( ii\right) $ Every left $\Gamma $-ideal of $S$ is $\Gamma $-semiprime.

$\left( iii\right) $ Every two-sided $\Gamma $-ideal of $S$ is $\Gamma $%
-semiprime

\begin{proof}
$(i):$ Let $R$ be a right $\Gamma $-ideal of an intra-regular $\Gamma $-AG$%
^{\ast \ast }$-groupoid $S$. Let $a^{2}\in R$ and let $a\in S.$ Now since $S$
is an intra-regular so there exists $x,$ $y\in S$ and $\beta ,\gamma ,\delta
\in \Gamma $ such that $a=(x\beta (a\delta a))\gamma y$. Now by using $(3),$ 
$(1)$ and $(2),$ we have%
\begin{eqnarray*}
a &=&(x\beta (a\delta a))\gamma y=(a\beta (x\delta a))\gamma y=(y\beta
(x\delta a))\gamma a=(y\beta (x\delta a))\gamma ((x\beta (a\delta a))\gamma
y) \\
&=&(x\beta (a\delta a))\gamma ((y\beta (x\delta a))\gamma y)=(x\beta (y\beta
(x\delta a)))\gamma ((a\delta a)\gamma y) \\
&=&(a\delta a)\gamma ((x\beta (y\beta (x\delta a)))\gamma y)\in R\Gamma
(S\Gamma S)=R\Gamma S\subseteq R.
\end{eqnarray*}

Which shows that $R$ is $\Gamma $-semiprime.

$(ii):$ Let $L$ be a left $\Gamma $-ideal of $S.$ Let $a^{2}\in L$ and let $%
a\in S$ now since $S$ is an intra-regular so there exists $x,$ $y\in S$ and $%
\beta ,\gamma ,\delta \in \Gamma $ such that $a=(x\beta (a\delta a))\gamma
y, $ then by using $(3),$ $(1)$ and $(4),$ we have%
\begin{eqnarray*}
a &=&(x\beta (a\delta a))\gamma y=(a\beta (x\delta a))\gamma y=(y\beta
(x\delta a))\gamma a \\
&=&(y\beta (x\delta a))\gamma ((x\beta (a\delta a))\gamma y)=(x\beta
(a\delta a))\gamma ((y\beta (x\delta a))\gamma y) \\
&=&(y\beta (y\beta (x\delta a)))\gamma ((a\delta a)\gamma x)=(a\delta
a)\gamma ((y\beta (y\beta (x\delta a)))\gamma x) \\
&=&(x\delta (y\beta (y\beta (x\delta a))))\gamma (a\gamma a)\in S\Gamma
L\subseteq L.
\end{eqnarray*}

Which shows that $L$ is $\Gamma $-semiprime.

$(iii)$ is obvious.
\end{proof}

\begin{theorem}
In a $\Gamma $-AG$^{\ast \ast }$-groupoid $S$, the following statements are
equivalent.
\end{theorem}

$\left( i\right) $ $S$ is intra-regular.

$(ii)$ Every right $\Gamma $-ideal of $S$ is $\Gamma $-semiprime.

\begin{proof}
$(i)$ $\Longrightarrow (ii)$ is obvious by Theorem \ref{RLT}.

$(ii)$ $\Longrightarrow (i):$ Let $S$ be an intra-regular $\Gamma $-AG$%
^{\ast \ast }$-groupoid. Let $R$ be any right $\Gamma $-ideal of $S$ such
that $R$ is $\Gamma $-semiprime. Since $a^{2}\Gamma S$ is a right $\Gamma $%
-ideal of $S$, therefore $a^{2}\Gamma S$ is $\Gamma $-semiprime. Now clearly 
$a^{2}\in a^{2}\Gamma S$ so $a\in a^{2}\Gamma S.$ Now let $\delta \in S$,\
then by using $(3)$ and $(2),$ we have%
\begin{eqnarray*}
a &\in &(a\delta a)\Gamma S=(a\delta a)\Gamma (S\Gamma S)=S\Gamma ((a\delta
a)\Gamma S)=(S\Gamma S)\Gamma ((a\delta a)\Gamma S) \\
&=&(S\Gamma (a\delta a))\Gamma (S\Gamma S)=(S\Gamma (a\delta a))\Gamma S.
\end{eqnarray*}

Which shows that $S$ is an intra regular.
\end{proof}

\begin{theorem}
\label{intra if rintl}A $\Gamma $-AG$^{\ast \ast }$-groupoid $S$ is
intra-regular if and only if $R\cap L=R\Gamma L$, for every $\Gamma $%
-semiprime right $\Gamma $-ideal $R$ and every left $\Gamma $-ideal $L$ of $%
S $.
\end{theorem}

\begin{proof}
Let $S$ be an intra-regular $\Gamma $-AG$^{\ast \ast }$-groupoid and $R$ and 
$L$ be right and left $\Gamma $-ideal of $S$ respectively, then by Theorem %
\ref{LisR}, $R$ and $L$ become $\Gamma $-ideals of $S$, therefore by Lemma %
\ref{ij}, $R\cap L=R\Gamma L,$ for every $\Gamma $-ideal $R$ and $L,$ also
by Theorem \ref{RLT}, $R$ is $\Gamma $-semiprime.

Conversely, assume that $R\cap L=R\Gamma L$ for every right $\Gamma $-ideal $%
R,$ which is $\Gamma $-semiprime and every left $\Gamma $-ideal $L$ of $S$.
Since $a^{2}\in a^{2}\Gamma S$, which is a right $\Gamma $-ideal of $S$ so
is $\Gamma $-semiprime which implies that $a\in a^{2}\Gamma S$. Now clearly $%
S\Gamma a$ is a left $\Gamma $-ideal of $S$ and $a\in S\Gamma a,$ therefore
let $\gamma \in \Gamma ,$ then by using $(4)$, $(1)$ and $(2),$ we have%
\begin{eqnarray*}
a &\in &\left( (a\gamma a)\Gamma S\right) \cap \left( S\Gamma a\right)
=(\left( a\gamma a)\Gamma S\right) \Gamma \left( S\Gamma a\right) \subseteq
\left( (a\gamma a)\Gamma S\right) \Gamma \left( S\Gamma S\right) \\
&=&\left( (a\gamma a)\Gamma S\right) \Gamma S=\left( \left( a\gamma a\right)
\Gamma S\right) \Gamma S=(S\Gamma S)\Gamma (a\gamma a) \\
&=&\left( S\Gamma a\right) \Gamma \left( S\Gamma a\right) =S\Gamma ((S\Gamma
a)\Gamma a)=\left( S\Gamma \left( a\gamma a\right) \right) \Gamma S\text{.}
\end{eqnarray*}

Therefore $S$ is an intra-regular.
\end{proof}

\begin{theorem}
For a $\Gamma $-AG$^{\ast \ast }$-groupoid $S$ , the following statements
are equivalent.

$\left( i\right) $ $S$ is intra-regular.

$\left( ii\right) $ $L\cap R\subseteq L\Gamma R$, for every right $\Gamma $%
-ideal $R,$ which is $\Gamma $-semiprime and every left $\Gamma $-ideal $L$
of $S$.

$\left( iii\right) $ $L\cap R\subseteq \left( LR\right) L$, for every $%
\Gamma $-semiprime right $\Gamma $-ideal $R$ and every left $\Gamma $-ideal $%
L$.
\end{theorem}

\begin{proof}
$\left( i\right) \Rightarrow \left( iii\right) :$ Let $S$ be an
intra-regular $\Gamma $-AG$^{\ast \ast }$-groupoid and $L,$ $R$ be any left
and right $\Gamma $-ideals of $S$ and let $k\in L\cap R,$ which implies that 
$k\in L$ and $k\in R$. Since $S$ is intra-regular so there exist $x$, $y$ in 
$S$, and $\alpha ,\beta ,\gamma \in \Gamma $ such that $k=\left( x\alpha
(k\gamma k)\right) \beta y$, then by using $(3)$, $(1)$ and $(4)$, we have%
\begin{eqnarray*}
k &=&\left( x\alpha \left( k\gamma k\right) \right) \beta y=\left( k\alpha
\left( x\gamma k\right) \right) \beta y=\left( y\alpha \left( x\gamma
k\right) \right) \beta k \\
&=&\left( y\alpha \left( x\gamma \left( \left( x\alpha (k\gamma k)\right)
\beta y\right) \right) \right) \beta k=\left( y\alpha \left( \left( x\alpha
(k\gamma k)\right) \gamma \left( x\beta y\right) \right) \right) \beta k \\
&=&\left( \left( x\alpha \left( k\gamma k\right) \right) \alpha \left(
y\gamma \left( x\beta y\right) \right) \right) \beta k=\left( \left( k\alpha
\left( x\gamma k\right) \right) \alpha \left( y\gamma \left( x\beta y\right)
\right) \right) \beta k \\
&\in &\left( (R\Gamma \left( S\Gamma L\right) )\Gamma S\right) \Gamma
L\subseteq \left( \left( R\Gamma L\right) \Gamma S\right) \Gamma L=\left(
L\Gamma S\right) \Gamma \left( R\Gamma L\right) \\
&=&\left( L\Gamma R\right) \Gamma \left( S\Gamma L\right) \subseteq \left(
L\Gamma R\right) \Gamma L\text{.}
\end{eqnarray*}

which implies that $L\cap R\subseteq \left( L\Gamma R\right) \Gamma L$. Also
by Theorem \ref{RLT}, $L$ is $\Gamma $-semiprime.

$\left( iii\right) \Rightarrow \left( ii\right) :$ Let $R$ and $L$ be any
left and right $\Gamma $-ideals of $S$ and $R$ is $\Gamma $-semiprime, then
by assumption $\left( iii\right) $ and by using $(4)$, $(3)$ and $(1)$, we
have%
\begin{eqnarray*}
R\cap L &\subseteq &\left( R\Gamma L\right) \Gamma R\subseteq \left( R\Gamma
L\right) \Gamma S=\left( R\Gamma L\right) \Gamma \left( S\Gamma S\right)
=\left( S\Gamma S\right) \Gamma \left( L\Gamma R\right) \\
&=&L\Gamma \left( \left( S\Gamma S\right) \Gamma R\right) =L\Gamma \left(
\left( R\Gamma S\right) \Gamma S\right) \subseteq L\Gamma \left( R\Gamma
S\right) \subseteq L\Gamma R\text{.}
\end{eqnarray*}

$\left( ii\right) \Rightarrow \left( i\right) :$ Since $a\in S\Gamma a,$
which is a left $\Gamma $-ideal of $S$, and $a^{2}\in a^{2}\Gamma S$, which
is a $\Gamma $-semiprime right $\Gamma $-ideal of $S$, therefore, $a\in
a^{2}\Gamma S$. Now by using $(4),$ we have%
\begin{eqnarray*}
a &\in &\left( S\Gamma a\right) \cap \left( a^{2}\Gamma S\right) \subseteq
\left( S\Gamma a\right) \Gamma (a^{2}\Gamma S)\subseteq \left( S\Gamma
S\right) \Gamma \left( a^{2}\Gamma S\right) \\
&=&\left( S\Gamma a^{2}\right) \Gamma \left( S\Gamma S\right) =\left(
S\Gamma a^{2}\right) \Gamma S\text{.}
\end{eqnarray*}

Hence $S$ is intra-regular.
\end{proof}

A $\Gamma $-AG$^{\ast \ast }$-groupoid $S$ is called $\Gamma $-totally
ordered under inclusion if $P$ and $Q$ are any $\Gamma $-ideals of $S$ such
that either $P\subseteq Q$ or $Q\subseteq P$.

A $\Gamma $-ideal $P$ of a $\Gamma $-AG$^{\ast \ast }$-groupoid $S$ is
called $\Gamma $-strongly irreducible if $A\cap B\subseteq P$ implies either 
$A\subseteq P$ or $B\subseteq P$, for all $\Gamma $-ideals $A$, $B$ and $P$
of $S$.

\begin{lemma}
Every $\Gamma $-ideal of an intra-regular $\Gamma $-AG$^{\ast \ast }$%
-groupoid $S$ is $\Gamma $-prime if and only if it is $\Gamma $-strongly
irreducible.
\end{lemma}

\begin{proof}
It is an easy consequence of Lemma \ref{ij}.
\end{proof}

\begin{theorem}
Every $\Gamma $-ideal of an intra-regular $\Gamma $-AG$^{\ast \ast }$%
-groupoid $S$ is $\Gamma $-prime if and only if $S$ is $\Gamma $-totally
ordered under inclusion.
\end{theorem}

\begin{proof}
Assume that every $\Gamma $-ideal of $S$ is $\Gamma $-prime. Let $P$ and $Q$
be any $\Gamma $-ideals of $S$, so by Lemma $\ref{ij}$, $P\Gamma Q=P\cap Q$,
and by Lemma \ref{idl}, $P\cap Q$ is a $\Gamma $-ideal of $S$, so is prime,
therefore $P\Gamma Q\subseteq P\cap Q,$ which implies that $P\subseteq P\cap
Q$ or $Q\subseteq P\cap Q,$ which implies that $P\subseteq Q$ or $Q\subseteq
P$. Hence $S$ is $\Gamma $-totally ordered under inclusion.

Conversely, assume that $S$ is $\Gamma $-totally ordered under inclusion.
Let $I$, $J$ and $P$ be any $\Gamma $-ideals of $S$ such that $I\Gamma
J\subseteq P$. Now without loss of generality assume that $I\subseteq J$
then 
\begin{equation*}
I=I\Gamma I\subseteq I\Gamma J\subseteq P\text{.}
\end{equation*}%
Therefore either $I\subseteq P$ or $J\subseteq P$, which implies that $P$ is 
$\Gamma $-prime.
\end{proof}

\begin{theorem}
The set of all $\Gamma $-ideals $I_{s}$ of an intra-regular $\Gamma $-AG$%
^{\ast \ast }$-groupoid $S$, forms a $\Gamma $-semilattice structure.
\end{theorem}

\begin{proof}
Let $A$, $B\in I_{s}$, since $A$ and $B$ are $\Gamma $-ideals of $S$, then
by using $(2),$ we have%
\begin{eqnarray*}
(A\Gamma B)\Gamma S &=&\left( A\Gamma B\right) \Gamma \left( S\Gamma
S\right) =\left( A\Gamma S\right) \Gamma \left( B\Gamma S\right) \subseteq
A\Gamma B\text{. } \\
\text{Also }S\Gamma (A\Gamma B) &=&\left( S\Gamma S\right) \Gamma \left(
A\Gamma B\right) =\left( S\Gamma A\right) \Gamma (S\Gamma B)\subseteq
A\Gamma B\text{.}
\end{eqnarray*}%
Thus $A\Gamma B$ is a $\Gamma $-ideal of $S$. Hence $I_{s}$ is closed. Also
using Lemma \ref{ij}, we have, $A\Gamma B=A\cap B=B\cap A=B\Gamma A$, which
implies that $I_{s}$ is commutative, so is associative. Now by using Theorem %
\ref{ii}, $A\Gamma A=A$, for all $A\in I_{s}$. Hence $I_{s}$ is $\Gamma $%
-semilattice.
\end{proof}

\begin{theorem}
A two-sided $\Gamma $-ideal of an intra-regular $\Gamma $-AG$^{\ast \ast }$%
-groupoid $S$ is minimal if and only if it is the intersection of two
minimal two-sided $\Gamma $-ideals.

\begin{proof}
Let $S$ be an intra-regular $\Gamma $-AG$^{\ast \ast }$-groupoid and $Q$ be
a minimal two-sided $\Gamma $-ideal of $S$, let $a\in Q$. As $S\Gamma
(S\Gamma a)\subseteq S\Gamma a$ and $S\Gamma (a\Gamma S)\subseteq a\Gamma
(S\Gamma S)=a\Gamma S,$ which shows that $S\Gamma a$ and $a\Gamma S$ are
left $\Gamma $-ideals of $S$ so by Lemma \ref{LisR}, $S\Gamma a$ and $%
a\Gamma S$ are two-sided $\Gamma $-ideals of $S$.

Now%
\begin{eqnarray*}
&&S\Gamma (S\Gamma a\cap a\Gamma S)\cap (S\Gamma a\cap a\Gamma S)\Gamma S \\
&=&S\Gamma (S\Gamma a)\cap S\Gamma (a\Gamma S)\cap (S\Gamma a)\Gamma S\cap
(a\Gamma S)\Gamma S \\
&\subseteq &(S\Gamma a\cap a\Gamma S)\cap (S\Gamma a)\Gamma S\cap S\Gamma
a\subseteq S\Gamma a\cap a\Gamma S.
\end{eqnarray*}

Which implies that $S\Gamma a\cap a\Gamma S$ is a $\Gamma $-quasi ideal so
by Theorems \ref{12} and \ref{quo}, $S\Gamma a\cap a\Gamma S$ is a two-sided 
$\Gamma $-ideal.

Also since $a\in Q$, we have%
\begin{equation*}
S\Gamma a\cap a\Gamma S\subseteq S\Gamma Q\cap Q\Gamma S\subseteq Q\cap
Q\subseteq Q\text{.}
\end{equation*}

Now since $Q$ is minimal so $S\Gamma a\cap a\Gamma S=Q,$ where $S\Gamma a$
and $a\Gamma S$ are minimal two-sided $\Gamma $-ideals of $S$, because let $%
I $ be a $\Gamma $-ideal of $S$ such that $I\subseteq S\Gamma a,$ then%
\begin{equation*}
I\cap a\Gamma S\subseteq S\Gamma a\cap a\Gamma S\subseteq Q,
\end{equation*}%
which implies that%
\begin{equation*}
I\cap a\Gamma S=Q.\text{ Thus }Q\subseteq I.
\end{equation*}

So we have%
\begin{eqnarray*}
S\Gamma a &\subseteq &S\Gamma Q\subseteq S\Gamma I\subseteq I,\text{ gives}
\\
S\Gamma a &=&I.
\end{eqnarray*}

Thus $S\Gamma a$ is a minimal two-sided $\Gamma $-ideal of $S$. Similarly $%
a\Gamma S$ is a minimal two-sided $\Gamma $-ideal of $S.$

Conversely, let $Q=I\cap J$ be a two-sided $\Gamma $-ideal of $S$, where $I$
and $J$ are minimal two-sided $\Gamma $-ideals of $S,$ then by Theorem \ref%
{12} and \ref{quo}, $Q$ is a $\Gamma $-quasi ideal of $S$, that is $S\Gamma
Q\cap Q\Gamma S\subseteq Q.$

Let $Q^{^{\prime }}$ be a two-sided $\Gamma $-ideal of $S$ such that $%
Q^{^{\prime }}\subseteq Q$, then

\begin{eqnarray*}
S\Gamma Q^{^{\prime }}\cap Q^{^{\prime }}\Gamma S &\subseteq &S\Gamma Q\cap
Q\Gamma S\subseteq Q,\text{ also }S\Gamma Q^{^{\prime }}\subseteq S\Gamma
I\subseteq I\text{ } \\
\text{and }Q^{^{\prime }}\Gamma S &\subseteq &J\Gamma S\subseteq J\text{.}
\end{eqnarray*}

Now 
\begin{eqnarray*}
S\Gamma \left( S\Gamma Q^{^{\prime }}\right) &=&\left( S\Gamma S\right)
\Gamma \left( S\Gamma Q^{^{\prime }}\right) =\left( Q^{^{\prime }}\Gamma
S\right) \Gamma \left( S\Gamma S\right) \\
&=&\left( Q^{^{\prime }}\Gamma S\right) \Gamma S=\left( S\Gamma S\right)
\Gamma Q^{^{\prime }}=S\Gamma Q^{^{\prime }}
\end{eqnarray*}%
implies that $S\Gamma Q^{^{\prime }}$ is a left $\Gamma $-ideal and hence a
two-sided $\Gamma $-ideal by Lemma \ref{LisR}. Similarly $Q^{^{\prime
}}\Gamma S$ is a two-sided $\Gamma $-ideal of $S$.

But since $I$ and $J$ are minimal two-sided $\Gamma $-ideals of $S$, so

\begin{equation*}
S\Gamma Q^{^{\prime }}=I\text{ and }Q^{^{\prime }}\Gamma S=J.
\end{equation*}

But $Q=I\cap J,$ which implies that,

\begin{equation*}
Q=S\Gamma Q^{^{\prime }}\cap Q^{^{\prime }}\Gamma S\subseteq Q^{^{\prime }}.
\end{equation*}

Which give us $Q=Q^{^{\prime }}$. Hence $Q$ is minimal.
\end{proof}
\end{theorem}

\end{document}